\documentclass [12pt,oneside,fullpage]{amsart}
\usepackage{amsmath}
\usepackage {amsfonts}
\usepackage {amsmath}
\usepackage {amsthm}
\usepackage {amssymb}
\usepackage {framed}
\usepackage {amsxtra}
\usepackage {enumerate}
\usepackage {graphicx}
\usepackage{color}
\usepackage{graphicx}
\usepackage{ams}

\usepackage[left=.6 in,right=.9 in, top=.9 in,footskip=0.3 in, bottom=0.3 in]{geometry}
\makeatletter

\newenvironment{figurehere}
  {\def\@captype{figure}}
  {}

\theoremstyle{definition}
\newtheorem{df}{Definition} [section]

\theoremstyle{plain}
\newtheorem{thm}[df]{Theorem}

\newtheorem{lemma}[df]{Lemma}

\newtheorem{problem}[df]{Problem}
\newtheorem{claim}[df]{Claim}

\title{Area Problems Involving Kasner Polygons}

\author{Dan Ismailescu}
\address{Department of Mathematics, Hofstra University,
Hempstead, NY 11549, USA.}
\email{matdpi@hofstra.edu}
\author{Minsuk Kim}
\address{Paul VI Catholic High School, 10675 Fairfax Blvd., Fairfax, VA 22030, USA.}
\email{mikey1115@naver.com}
\author{Kyung Jae Lee}
\address{Bergen County Academies, 200 Hackensack Avenue, Hackensack, NJ 07601, USA.}
\email{kyulee@bergen.org}
\author{Seong Hoon Lee}
\address{Philips Exeter Academy, 20 Main Street, Exeter, NH 03833, USA.}
\email{slee@exeter.edu}
\author{Taehyeun Park}
\address{Bronx High School of Science, 75 West 205th Street, Bronx, NY 10468, USA.}
\email{taehyun714@gmail.com}

\begin{document}

\begin{abstract}
Sequences of polygons generated by performing iterative processes on
an initial polygon have been studied extensively. One of the most
popular sequences is the one sometimes referred to as {\it Kasner
polygons}. Given a polygon $K$, the first Kasner descendant $K'$ of
$K$ is obtained by placing the vertices of $K'$ at the midpoints of
the edges of $K$.

More generally, for any fixed $m$ in $(0,1)$ one may define a
sequence of polygons $\{K^{t}\}_{t\ge 0}$ where each polygon
$K^{t}$ is obtained by dividing every edge of $K^{t-1}$ into the
ratio $m:(1-m)$ in the counterclockwise (or clockwise) direction and
taking these division points to be the vertices of $K^{t}$.

We are interested in the following problem

{\it Let $m$ be a fixed number in $(0,1)$ and let $n\ge 3$ be a
fixed integer. Further, let $K$ be a convex $n$-gon and denote by
$K'$, the first $m$-Kasner descendant of $K$, that is, the vertices
of $K'$ divide the edges of $K$ into the ratio $m:(1-m)$. What can
be said about the ratio between the area of $K'$ and the area of
$K$, when $K$ varies in the class of convex $n$-gons?}

We provide a complete answer to this question.

\end{abstract}
\maketitle


\begin{section}{\bf Introduction} Start with a fixed number $m$ in
$(0,1)$ and a convex $n$-gon $K$. Let $K'$ be the convex $n$-gon
whose vertices divide the edges of $K$ into the ratio $m:(1-m)$ in
the counterclockwise direction. We say $K'$ is the first $m$-Kasner
descendant of $K$. In general, we may construct a sequence of
polygons $\{K^{t}\}_{t\ge 0}$ where $K^0=K$, $K^1=K'$ and $K^{t+1}$
is the first $m$-Kasner descendant of $K^{t}$.

Kasner noticed that if $n=4$ and $m=1/2$ then $K'$ is always a
parallelogram. In \cite {K} he and his students managed to
characterize those $n$-gons $P$ which have the property that $P=K'$
for some convex $n$-gon $K$ and $m=1/2$.

It is easy to notice that all $n$-gons in the sequence $\{K^{t}\}$
defined above have the same centroid. This was proved repeatedly;
see e. g. \cite {D} for a matrix proof or \cite {Sch} for a proof
using Fourier series.

It has been shown by L\"{u}k\~{o} \cite {L} that the sequence
$\{K^{t}\}$ converges to an (affine) regular $n$-gon, thus proving a
conjecture of Fejes T\'{o}th \cite {LFT}. More on Kasner polygons
can be found in \cite{BS,C,HZ,KN, KM}.

In this paper we study the following:
\begin{problem}
Let $m$ be a fixed number in $(0,1)$ and let $n\ge 3$ be a fixed
integer. Let $K$ be a convex $n$-gon and denote by $K'$, the first
$m$-Kasner descendant of $K$. What can be said about
$\Delta(K')/\Delta(K)$, the ratio between the area of $K'$ and the
area of $K$, when $K$ varies in the class of convex $n$-gons?
\end{problem}

\noindent{\bf The Main Technique.} Throughout the entire paper we
use the {\it wedge product} of two vectors to express areas. This
operation, also known as {\it exterior product}, is defined as
follows. For any two vectors $\mathbf{v}=(a, \, b)$ and
$\mathbf{u}=(c,\,d)$ let the wedge product of $\mathbf{v}$ and
$\mathbf{u}$ be given by $\mathbf{v}\wedge\mathbf{u}:=(ad-bc)/2.$ It
is easy to see that the wedge product represents the {\it signed
area} of the triangle determined by the vectors $\mathbf{v}$ and
$\mathbf{u}$, where the $\pm$ sign depends on whether the angle
between $\mathbf{v}$ and $\mathbf{u}$ - measured in the
counterclockwise direction from $\mathbf{v}$ towards $\mathbf{u}$ -
is smaller than or greater than $180^{\circ}$.

The following properties of the wedge product are simple
consequences of the definition:
\begin{itemize}
\item{anti-commutativity: $\mathbf{v}\wedge\mathbf{u}=-
\mathbf{u}\wedge\mathbf{v}$ and in particular
$\mathbf{v}\wedge\mathbf{v}=0$.}

\item{linearity: $(\alpha\mathbf{v}+\beta
\mathbf{u})\wedge\mathbf{w}=\alpha\mathbf{v}\wedge\mathbf{w}
+\beta\mathbf{u}\wedge\mathbf{w}.$}
\end{itemize}

\end{section}

\begin{section}{\bf The Triangle Case}

\noindent Let $m$ be a fixed number in $(0,\,1)$ and let $K=ABC$ be
an arbitrary triangle. Construct points $M$, $N$ and $P$ on the
sides $AB$, $BC$ and $AC$, such that $AM:MB=BN:NC=CP:PA=m:(1-m).$
\noindent We call triangle $K'=MNP$ to be the {\it first $m$-Kasner
descendant of $K$}.

\begin{thm}\label{thmtriangle}
With the notations above we have that
\begin{equation}\label{eqtriangle}
\frac{\Delta(K')}{\Delta(K)}=1-3m(1-m).
\end{equation}
\end{thm}
\begin{proof}
Denote $\overrightarrow{AB}=\mathbf{v}_1$,
$\overrightarrow{BC}=\mathbf{v}_2$ and
$\overrightarrow{CA}=\mathbf{v}_3$ as in figure \ref{tdk1}.
Obviously, $\mathbf{v}_1+\mathbf{v}_2+\mathbf{v}_3=\mathbf{0}.$
After an appropriate scaling we may assume that the area of $ABC$ is equal to $1$, that is
\begin{equation*}
\Delta(K)=\mathbf{v}_1\wedge\mathbf{v}_2=\mathbf{v}_2\wedge\mathbf{v}_3=\mathbf{v}_3\wedge\mathbf{v}_1=1.
\end{equation*}

\begin{figurehere}
\begin{center}
\scalebox{.5}{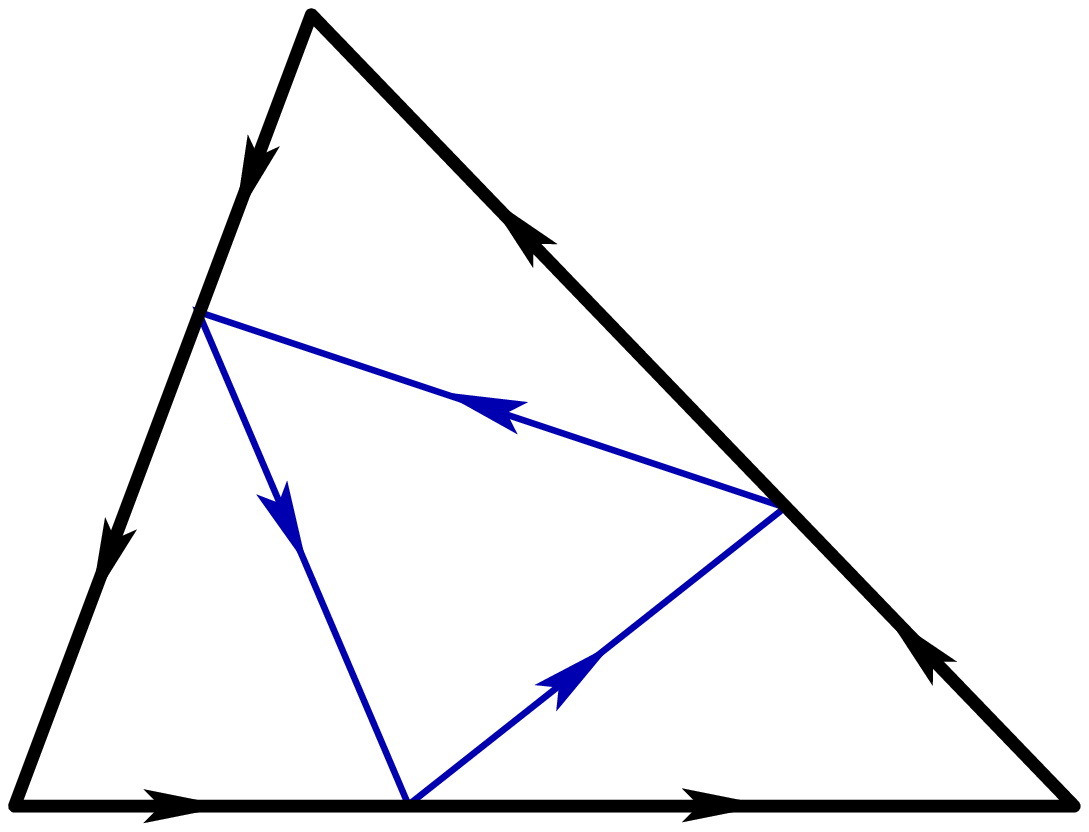} \caption{A triangle and its
first $m$-Kasner descendant} \label{tdk1}
\end{center}
\end{figurehere}

It is easy to see that $\overrightarrow{MN}=\overrightarrow{MB}+\overrightarrow{BN}=(1-m)\mathbf{v}_1+m\mathbf{v}_2$
and $\overrightarrow{NP}=\overrightarrow{NC}+\overrightarrow{CP}=(1-m)\mathbf{v}_2+m\mathbf{v}_3$.
It follows that
\begin{eqnarray*}
\Delta(K')&=&\overrightarrow{MN}\wedge\overrightarrow{NP}=((1-m)\mathbf{v}_1+m\mathbf{v}_2)\wedge((1-m)\mathbf{v}_2+m\mathbf{v}_3)=\\
&=&(1-m)^2(\mathbf{v}_1\wedge\mathbf{v}_2)+m(1-m)(\mathbf{v}_1\wedge\mathbf{v}_3)+m^2(\mathbf{v}_2\wedge\mathbf{v}_3)=\\
&=&(1-m)^2+m(1-m)(-1)+m^2=1-3m(1-m).
\end{eqnarray*}
\end{proof}
\end{section}
\begin{section}{\bf The Quadrilateral Case}

\noindent Let $m$ be a fixed number in $(0,\,1)$ and let $K=ABCD$ be
an arbitrary quadrilateral. Construct points $M$, $N$, $P$ and $Q$
on the sides $AB$, $BC$, $CD$ and $DA$, such that
$AM:MB=BN:NC=CP:PD=DQ:QA=m:(1-m).$ We call quadrilateral $K'=MNPQ$
to be the {\it first $m$-Kasner descendant of $K$}.
\begin{thm}\label{thmquadrangle}
With the notations above we have that
\begin{equation}\label{eqquadrangle}
\frac{\Delta(K')}{\Delta(K)}=1-2m(1-m).
\end{equation}
\end{thm}
\begin{proof}
Denote $\overrightarrow{AB}=\mathbf{v}_1$,
$\overrightarrow{BC}=\mathbf{v}_2$,
$\overrightarrow{CD}=\mathbf{v}_3$ and
$\overrightarrow{DA}=\mathbf{v}_4$ as in figure \ref{tdk2}.
Obviously,
$\mathbf{v}_1+\mathbf{v}_2+\mathbf{v}_3+\mathbf{v}_4=\mathbf{0}.$
One can express the area of $ABCD$ in a couple of ways as below.
\begin{eqnarray}
\Delta(K)&=&\Delta(ABC)+\Delta(CDA)=\mathbf{v}_1\wedge\mathbf{v}_2+\mathbf{v}_3\wedge\mathbf{v}_4\qquad {\mbox {and}}\label{area1}\\
\Delta(K)&=&\Delta(BCD)+\Delta(CDA)=\mathbf{v}_2\wedge\mathbf{v}_3+\mathbf{v}_4\wedge\mathbf{v}_1.\label{area2}
\end{eqnarray}
\begin{figurehere}
\begin{center}
\scalebox{.4}{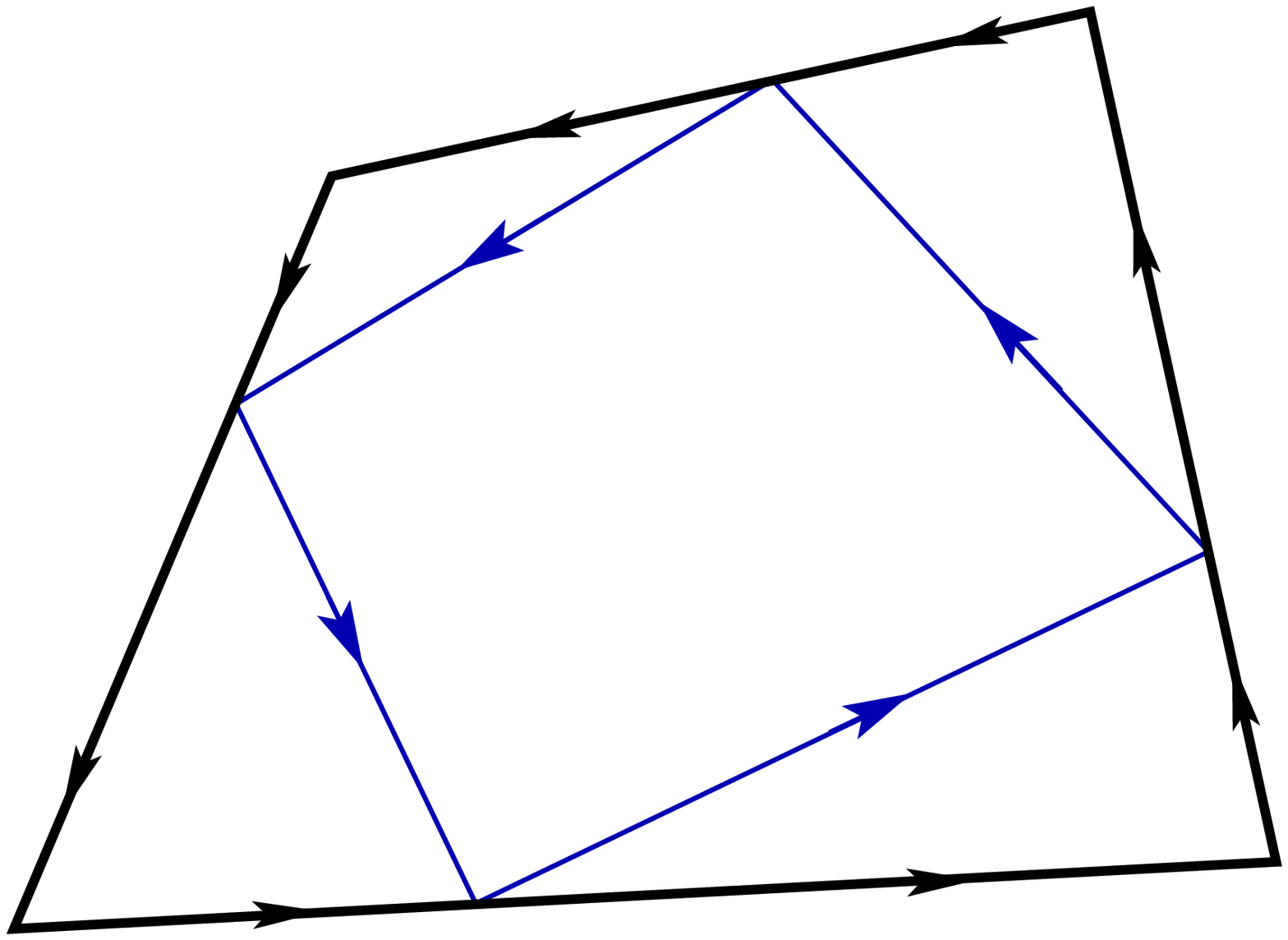} \caption{A convex
quadrilateral and its first $m$-Kasner descendant} \label{tdk2}
\end{center}
\end{figurehere}

On the other hand
\begin{eqnarray*}
\Delta(MNP)&=&\overrightarrow{MN}\wedge\overrightarrow{NP}=((1-m)\mathbf{v}_1+m\mathbf{v}_2)\wedge((1-m)\mathbf{v}_2+m\mathbf{v}_3)=\\
&=&(1-m)^2(\mathbf{v}_1\wedge\mathbf{v}_2)+m(1-m)(\mathbf{v}_1\wedge\mathbf{v}_3)+m^2(\mathbf{v}_2\wedge\mathbf{v}_3).\\
\Delta(PQM)&=&\overrightarrow{PQ}\wedge\overrightarrow{QM}=((1-m)\mathbf{v}_3+m\mathbf{v}_4)\wedge((1-m)\mathbf{v}_4+m\mathbf{v}_1)=\\
&=&(1-m)^2(\mathbf{v}_3\wedge\mathbf{v}_4)+m(1-m)(\mathbf{v}_3\wedge\mathbf{v}_1)+m^2(\mathbf{v}_4\wedge\mathbf{v}_1). \end{eqnarray*}
Adding the two equalities above term by term and using (\ref {area1}) and (\ref {area2}) we obtain
\begin{eqnarray*}
\Delta(K')&=&(1-m)^2((\mathbf{v}_1\wedge\mathbf{v}_2)+(\mathbf{v}_3\wedge\mathbf{v}_4))+m^2((\mathbf{v}_2\wedge\mathbf{v}_3)+(\mathbf{v}_4\wedge\mathbf{v}_1))=\\
&=&(1-m)^2\Delta(K)+m^2\Delta(K)=(1-2m(1-m))\Delta(K).
\end{eqnarray*}
\end{proof}
\noindent We have seen that the ratio between the area of a convex
$n$-gon and the area of its first $m$-Kasner descendant is constant
if $n\le 4$. This is not true anymore if the initial polygon has at
least five sides. In this later case we will be interested in the
range of values the ratio $\Delta(K')/\Delta(K)$ takes when $K$
belongs to the class of convex $n$-gons. We are investigating this
question in the following sections.
\end{section}

\begin{section}{\bf The Pentagon Case - A First Attempt}

\noindent Let $m$ be a fixed number in $(0,\,1)$ and let $K=ABCDE$
be an arbitrary convex pentagon. Construct points $M$, $N$, $P$, $Q$
and $R$ on the sides $AB$, $BC$, $CD$, $DE$ and $EA$ such that
$AM:MB=BN:NC=CP:PD=DQ:QE=ER:RA=m:(1-m).$ As before, we call the
pentagon $K'=MNPQR$ to be the {\it first $m$-Kasner descendant of
$K$}.

\begin{thm}\label{thmpentagon}
With the notations above we have that
\begin{equation}\label{eqpentagon}
1-2m(1-m)<\frac{\Delta(K')}{\Delta(K)}<1-m(1-m).
\end{equation}
Moreover, both lower and upper bounds are the best possible.
\end{thm}

For reasons which will become clear soon, we postpone the proof of
theorem \ref{thmpentagon} until the next section. Let us first try
to approach this problem the same way we proved theorems \ref
{thmtriangle} and \ref {thmquadrangle}. As before, denote
$\overrightarrow{AB}=\mathbf{v}_1$,
$\overrightarrow{BC}=\mathbf{v}_2$,
$\overrightarrow{CD}=\mathbf{v}_3$,
$\overrightarrow{DE}=\mathbf{v}_4$ and
$\overrightarrow{EA}=\mathbf{v}_5$ as in figure \ref{tdk3}.
Obviously,
$\mathbf{v}_1+\mathbf{v}_2+\mathbf{v}_3+\mathbf{v}_4+\mathbf{v}_5=\mathbf{0}.$

\begin{figurehere}
\begin{center}
\scalebox{.5}{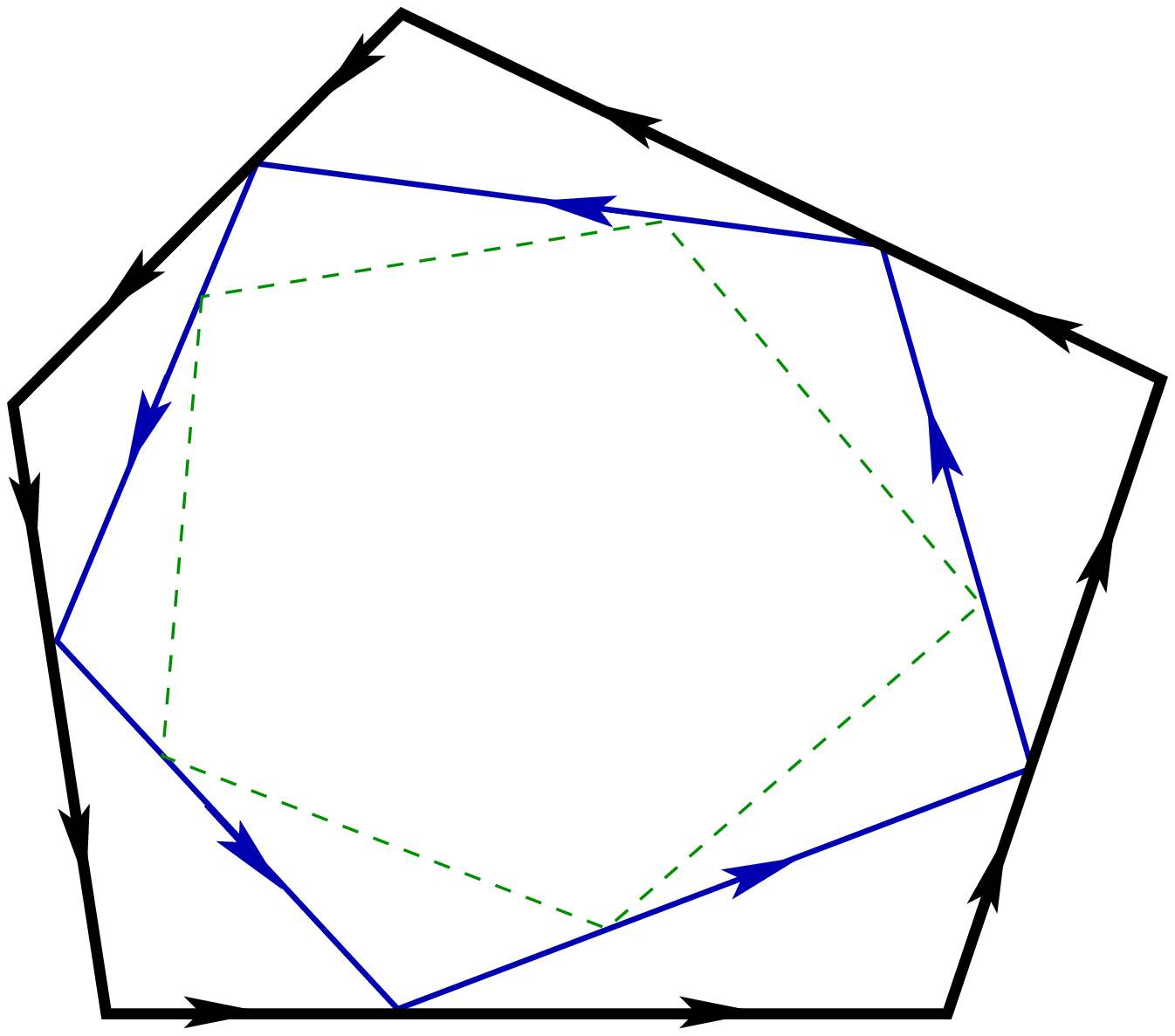} \caption{A convex
pentagon and its first two $m$-Kasner descendants} \label{tdk3}
\end{center}
\end{figurehere}

Let us introduce a couple of notations which are going to be useful later.
For every $1\le i,\,j \le 5$ denote $a_{ij}=\mathbf{v}_i\wedge\mathbf{v}_j$.
Moreover, let
\begin{eqnarray}
S:=a_{12}+a_{23}+a_{34}+a_{45}+a_{51}\label{S}.\\
T:=a_{13}+a_{24}+a_{35}+a_{41}+a_{52}\label{T}.
\end{eqnarray}
It is easy to see that $\Delta(MBN)=\overrightarrow{MB}\wedge\overrightarrow{BN}=(1-m)\mathbf{v}_1\wedge m\mathbf{v}_2=m(1-m)a_{12}$. Similar expressions can be found for the areas of the other four triangles
$NCP$, $PDQ$, $QER$ and $RAM$. Using (\ref {S}) it follows immediately that
\begin{equation}\label{delta1}
\Delta(K')=\Delta(K)-m(1-m)(a_{12}+a_{23}+a_{34}+a_{45}+a_{51})=\Delta(K)-m(1-m)S.
\end{equation}
Notice that there are several different ways in which $\Delta(K)$ can be expressed in terms of the $a_{ij}$-s.
For instance, we have that
\begin{eqnarray}\label{a13}
\Delta(K)&=&\Delta(ABC)+\Delta(ACD)+\Delta(ADE)=\overrightarrow{AB}\wedge\overrightarrow{BC}
+\overrightarrow{AC}\wedge\overrightarrow{CD}+\overrightarrow{DE}\wedge\overrightarrow{EA}=\nonumber\\
&=&\mathbf{v}_1\wedge\mathbf{v}_2+(\mathbf{v}_1+\mathbf{v}_2)\wedge\mathbf{v}_3+\mathbf{v}_4\wedge\mathbf{v}_5
=a_{12}+a_{13}+a_{23}+a_{45}.
\end{eqnarray}
Using equalities (\ref {delta1}) and (\ref {a13}) we easily derive the following
\begin{equation*}
\frac{\Delta(K')}{\Delta(K)}=1-m(1-m)\cdot\frac{a_{12}+a_{23}+a_{34}+a_{45}+a_{51}}{a_{12}+a_{13}+a_{23}+a_{45}}.
\end{equation*}
{\bf Observation.}
Proving theorem \ref {thmpentagon} reduces to showing that
\begin{equation*}
1<\frac{a_{12}+a_{23}+a_{34}+a_{45}+a_{51}}{a_{12}+a_{13}+a_{23}+a_{45}}<2
\end{equation*}
and that these inequalities cannot be improved.
This is somewhat of an awkward task. The reason is that the $a_{ij}$-s are not independent quantities.
Indeed, on one hand we have that $\mathbf{v}_1+\mathbf{v}_2+\mathbf{v}_3+\mathbf{v}_4+\mathbf{v}_5=\mathbf{0}.$
This implies for instance that $\mathbf{v}_1\wedge(\mathbf{v}_1+\mathbf{v}_2+\mathbf{v}_3+\mathbf{v}_4+\mathbf{v}_5)=
a_{12}+a_{13}+a_{14}+a_{15}=0$.

On the other hand, it can be easily shown that for any four distinct indices
$i$, $j$, $k$ and $l$ in $\{1,\, 2,\,3,\,4,\,5\}$ we have the following equality, known as Pl\"{u}cker's identity
\begin{equation}\label{plucker}
a_{ij}a_{kl}-a_{ik}a_{jl}+a_{il}a_{jk}=0.
\end{equation}
Indeed, it is easy to see that among the four vectors $\mathbf{v}_i$, $\mathbf{v}_j$, $\mathbf{v}_k$ and $\mathbf{v}_l$ there are two which are independent; say $\mathbf{v}_i$ and $\mathbf{v}_j$ are those vectors. Then $\mathbf{v}_k$ and $\mathbf{v}_l$ can be expressed as linear combinations of $\mathbf{v}_i$ and $\mathbf{v}_j$. Suppose that $\mathbf{v}_k=\alpha\mathbf{v}_i+\beta\mathbf{v}_j$ and $\mathbf{v}_l=\lambda\mathbf{v}_i+\mu\mathbf{v}_j$.
It follows immediately that $a_{ik}=\mathbf{v}_i\wedge(\alpha\mathbf{v}_i+\beta\mathbf{v}_j)=\beta a_{ij}$,
$a_{il}=\mathbf{v}_i\wedge(\lambda\mathbf{v}_i+\mu\mathbf{v}_j)=\mu a_{ij}$,
$a_{jk}=\mathbf{v}_j\wedge(\alpha\mathbf{v}_i+\beta\mathbf{v}_j)=-\alpha a_{ij}$,
$a_{jl}=\mathbf{v}_j\wedge(\lambda\mathbf{v}_i+\mu\mathbf{v}_j)=-\lambda a_{ij}$ and finally,
$a_{kl}=(\alpha\mathbf{v}_i+\beta\mathbf{v}_j)\wedge(\lambda\mathbf{v}_i+\mu\mathbf{v}_j)=(\alpha\mu-\beta\lambda) a_{ij}$. Substituting these equalities into the left side of (\ref {plucker}) we obtain the desired identity.

While it seems that this line of attack is destined to failure, one
can still derive an interesting fact. Let $K'':=STUVW$ be {\it the
second $m$-Kasner descendant of $K$} - see figure \ref {tdk3}. We
would like to see whether there is a relationship linking the areas
of $K$, $K'$ and $K''$.

For $1\le i,\,j\le 5$ denote $\mathbf{v}'_i=(1-m)\mathbf{v}_i+m\mathbf{v}_{i+1}$, $a'_{ij}=\mathbf{v}'_i\wedge\mathbf{v}'_j$ and $S':=a'_{12}+a'_{23}+a'_{34}+a'_{45}+a'_{51}$.
First notice that
\begin{equation*}
a'_{i,i+1}=\left[(1-m)\mathbf{v}_i+m\mathbf{v}_{i+1}\right]
\wedge\left[(1-m)\mathbf{v}_{i+1}+m\mathbf{v}_{i+2}\right]=(1-m)^2a_{i,i+1}+m(1-m)a_{i,i+2}+m^2a_{i+1,i+2}.
\end{equation*}
Using notations (\ref {S}) and (\ref {T}) it follows that
\begin{equation}\label{S'}
S'=\sum_{i=1}^5 a'_{i,i+1}=(1-m)^2S+m(1-m)T+m^2S=(1-2m(1-m))S+m(1-m)T.
\end{equation}
A reasoning similar to the one which led us to equality (\ref {delta1}) can be used to show that
\begin{equation}\label{delta2}
\Delta(K'')=\Delta(K')-m(1-m)S'.
\end{equation}
Also, relation (\ref {a13}) can be rewritten as
$a_{13}=\Delta(K)-a_{12}-a_{23}-a_{34}.$ If one expresses the area
of $K$ as
$\Delta(K)=\Delta(BCD)+\Delta(BDE)+\Delta(BEA)=a_{23}+a_{24}+a_{34}+a_{51}$
we get that
$a_{24}=\Delta(K)-a_{23}-a_{34}-a_{51}.$ In an analogous manner one
can obtain expressions for $a_{35}$, $a_{41}$ and $a_{52}$. By
adding these relations term by term and taking into account (\ref
{T}) we obtain that
\begin{equation}\label{TSD}
T=a_{13}+a_{24}+a_{35}+a_{41}+a_{52}=5\Delta(K)-3S.
\end{equation}
By eliminating the quantities $S'$, $S$ and $T$ between equalities (\ref {delta1}), (\ref{S'}), (\ref {delta2})
and (\ref {TSD}) we finally obtain a linear relationship linking the areas of $K$, $K'$ and $K''$.
\begin{equation}\label{recurrence}
\Delta(K'')=(2-5r)\Delta(K')-(1-5r+5r^2)\Delta(K),\qquad\qquad
{\mbox{where}}\,\,\, r:=m(1-m).
\end{equation}
In other words, if we know the area of the initial pentagon and the
area of its first $m$-Kasner descendant we can compute the area of
any of the $m$-Kasner descendants, $K^{t}$, for any $t\ge 2$.
Similar recurrence relationships are valid for polygons with more
than five sides. In general, for convex $n$-gons, the recurrence
involves $\lceil(n+1)/2\rceil$ consecutive $m$-Kasner descendants.
We omit the details.

\end{section}
\begin{section}{\bf The Pentagon Case - A Second (and Successful) Approach}

\noindent Let $K=ABCDE$ be an arbitrary convex pentagon and let $m$ be a fixed
constant in $(0,\,1)$. Denote $r:=m(1-m)$. After an
eventual relabeling of the vertices we may assume that
\begin{equation}\label{ABC}
\Delta(ABC)=\min\{\Delta(ABC),\, \Delta(BCD),\, \Delta(CDE),\,
\Delta(DEA),\, \Delta(EAB)\}
\end{equation}
\noindent In the literature, such triangles formed by three
consecutive vertices of a convex polygon are sometimes called {\it
ears}. Assumption (\ref {ABC}) above fixes the ear of least area.
Denote the intersection of $AD$ and $CE$ by $O$. Then define
$\mathbf{v}_1 = \overrightarrow{OA}$, $\mathbf{v}_2 =
\overrightarrow{OC}$. After an appropriate scaling, we may assume
that $\mathbf{v}_1\wedge \mathbf{v}_2 = \Delta(AOC) = 1$.

Since $E$, $O$, and $C$ are collinear and $D$, $O$, and $A$ are
collinear, we can write $\overrightarrow{DO} = a \cdot
\overrightarrow{OA} = a \mathbf{v}_1$ and $\overrightarrow{EO} = b
\cdot \overrightarrow{OC} = b \mathbf{v}_2$, with $a,\ b > 0$ (see
figure \ref{tdk4}).

\begin{figurehere}
\begin{center}
\scalebox{.5}{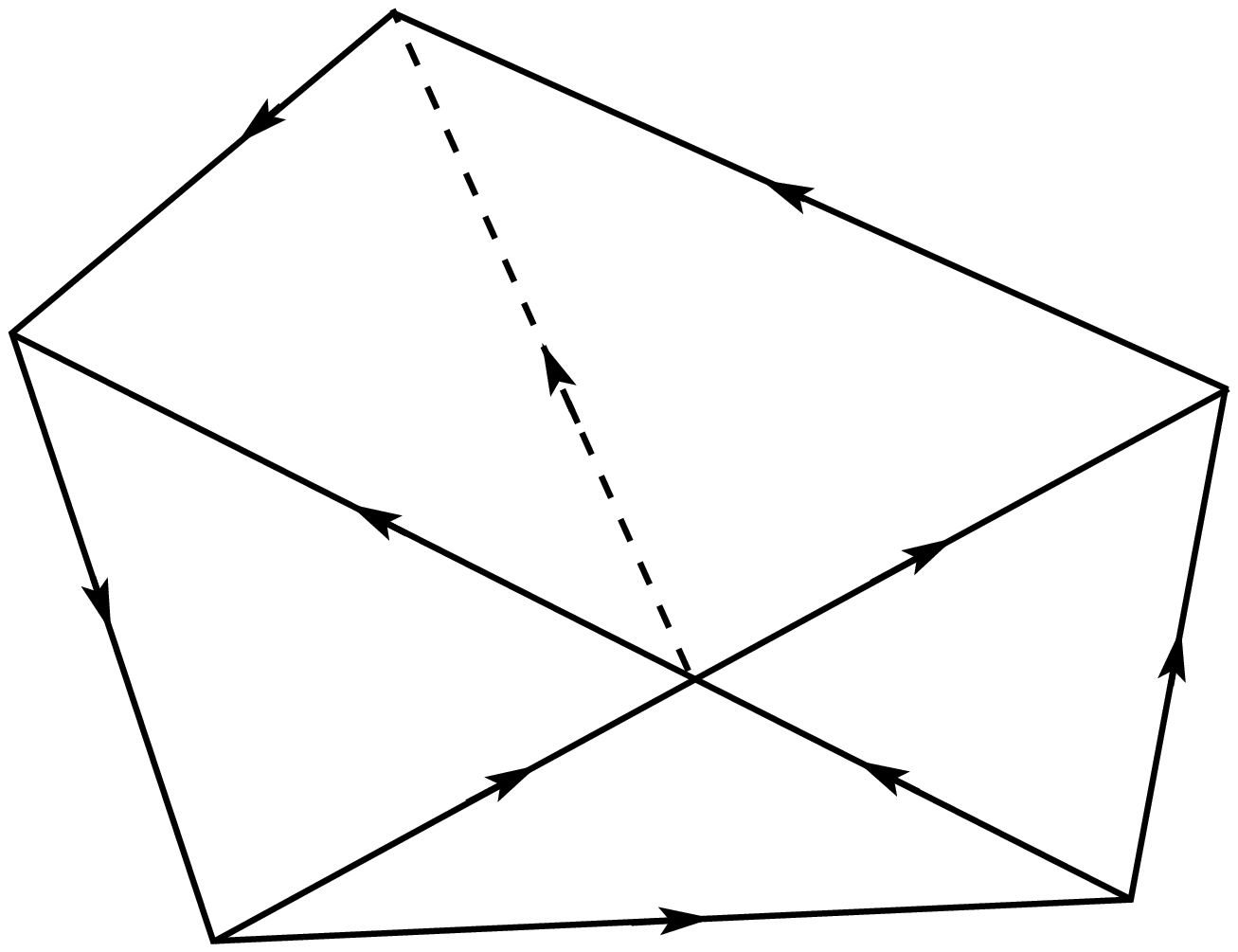} \caption{A better setup
for the convex pentagon problem } \label{tdk4}
\end{center}
\end{figurehere}

\noindent Using the triangle rule, we obtain that
$\overrightarrow{CD} = -a\mathbf{v}_1 - \mathbf{v}_2$,
$\overrightarrow{DE} = a\mathbf{v}_1 - b\mathbf{v}_2$, and
$\overrightarrow{EA} = \mathbf{v}_1 +b\mathbf{v}_2$.\\
\noindent We know that every vector in the plane can be written as a
linear combination of any two independent vectors. Set
$\overrightarrow{OB} =\mathbf{v}_3 = c \mathbf{v}_1 + d
\mathbf{v}_2$ - refer again to figure \ref{tdk4}. We also know that
$\overrightarrow{AB} = \mathbf{v}_3 - \mathbf{v}_1$ and
$\overrightarrow{BC} = \mathbf{v}_2 - \mathbf{v}_3$. We have that
\begin{eqnarray*}
\Delta(OAB) &=& \mathbf{v}_1\wedge \mathbf{v}_3 =
\mathbf{v}_1\wedge (c\mathbf{v}_1 + d\mathbf{v}_2) = d,\\
\Delta(OBC)&=& \mathbf{v}_3\wedge\mathbf{v}_2 = (c\mathbf{v}_1 +
d\mathbf{v}_2)\wedge \mathbf{v}_2 = c.
\end{eqnarray*}
After similar calculations, we can write the areas of various
triangles in pentagon $ABCDE$ in terms of the positive constants
$a,\,b,\,c,\,d$ as shown below:
 $\Delta(OCD) =a\mathbf{v}_1\wedge\mathbf{v}_2 = a, \,\,\Delta(OEA) =
a\mathbf{v}_1\wedge b\mathbf{v}_2 =ab,\,\,\Delta(ODE)=
\mathbf{v}_1\wedge b\mathbf{v}_2 =b.$ We can now compute the total
area of the pentagon.

$\Delta(ABCDE) = \Delta(OAB) + \Delta(OBC) + \Delta(OCD) +
\Delta(ODE) + \Delta(OEA)$, that is,
\begin{equation}\label{ABCDE}
\Delta(K)=\Delta(ABCDE) = a+b+c+d+ab.
\end{equation}
\noindent Next, we compute the areas of the ears of the pentagon.
\begin{eqnarray}
\Delta(ABC)&=&\overrightarrow{AB}\wedge\overrightarrow{BC}=(\mathbf{v}_3-\mathbf{v}_1)\wedge(\mathbf{v}_2-\mathbf{v}_3)=c+d-1,\label{ABC1}\\
\Delta(BCD)&=&\overrightarrow{BC}\wedge\overrightarrow{CD}=(\mathbf{v}_2-\mathbf{v}_3)\wedge(-a\mathbf{v}_1-\mathbf{v}_2)=a-ad+c,\label{BCD}\\
\Delta(CDE)&=&\overrightarrow{CD}\wedge\overrightarrow{DE}=(-a\mathbf{v}_1-\mathbf{v}_2)\wedge(a\mathbf{v}_1-b\mathbf{v}_2)=ab+a,\nonumber\\
\Delta(DEA)&=&\overrightarrow{DE}\wedge\overrightarrow{EA}=(a\mathbf{v}_1-b\mathbf{v}_2)\wedge(\mathbf{v}_1+b\mathbf{v}_2)=ab+b,\nonumber\\
\Delta(EAB)&=&\overrightarrow{EA}\wedge\overrightarrow{AB}=(\mathbf{v}_1+b\mathbf{v}_2)\wedge(\mathbf{v}_3-\mathbf{v}_1)=b-bc+d.\label{EAB}
\end{eqnarray}
It follows that
\begin{equation}\label{areaears}
\Delta(ABC)+\Delta(BCD)+\Delta(CDE)+\Delta(DEA)+\Delta(EAB)=2(a+b+c+d+ab)-1-ad-bc.
\end{equation}
Consider now $K'$, the first $m$-Kasner descendant of the initial
pentagon. We did not include $K'$ in figure \ref{tdk4} in order to
keep things clear. However, it is easy to see from figure \ref{tdk3}
that the area of $K'$ is the difference between the area of $K$ and
the sum of the areas of the ears of the pentagon multiplied by a
factor of $r=m(1-m)$. This means that
\begin{equation}\label{ear_reasoning}
\Delta(K')=\Delta(K)-r(\Delta(ABC)+\Delta(BCD)+\Delta(CDE)+\Delta(DEA)+\Delta(EAB))
\end{equation}
which after making use of (\ref{ABCDE}) and (\ref{areaears}) becomes
$\Delta(K')=(1-2r)\Delta+r(1+ad+bc)$ and finally
\begin{equation}\label{goodratio}
\frac{\Delta(K')}{\Delta(K)}=1-2r+r\frac{1+ad+bc}{a+b+c+d+ab}.
\end{equation}
In order to prove theorem \ref {thmpentagon} it is enough to show that
\begin{claim}
With the notations from the present section we have
\begin{equation}\label{claim}
0<\frac{1+ad+bc}{a+b+c+d+ab}<1
\end{equation}
and none of these inequalities can be improved.
\begin{proof}
It is clear the ratio is greater than $0$ as $a$, $b$, $c$ and $d$ are all positive.
To show that the ratio can be arbitrarily close to $0$ take $a=b=n$ and $c=d=1$.
Then,
\begin{equation*}
\frac{1+ad+bc}{a+b+c+d+ab}=\frac{2n+1}{n^2+2n+2}\rightarrow 0 \,\,\,{\mbox {as}}\,\, n\rightarrow\infty.
\end{equation*}
To show that the ratio can be arbitrarily close to $1$ take $a=n$, $b=1/n$ and $c=d=1$.
For these choices
\begin{equation*}
\frac{1+ad+bc}{a+b+c+d+ab}=\frac{n^2+n+1}{n^2+3n+1}\rightarrow 1 \,\,\,{\mbox {as}}\,\, n\rightarrow\infty.
\end{equation*}
Remains to show that the ratio is always less than $1$.
Recall that assumption (\ref {ABC}) stated that triangle $ABC$ is the ear of the smallest area.
Refer first to equality (\ref {ABC1}). Since $\Delta(ABC)>0$ it follows that $c+d>1$.
On the other hand $\Delta(ABC)\le \Delta(BCD)$ which after using (\ref{ABC1}) and (\ref{BCD}) gives that
$c+d-1\le a-ad+c\Leftrightarrow(a+1)(1-d)\ge 0\Leftrightarrow d\le 1.$
Finally $\Delta(ABC)\le \Delta(EAB)$ which after using (\ref{ABC1}) and (\ref{EAB}) implies that
$c+d-1\le b-bc+d\Leftrightarrow (b+1)(1-c)\ge 0\Leftrightarrow c\le 1.$

The last three inequalities ($c+d>1$, $c\le 1$ and $d\le 1$) imply that
\begin{equation*}
1+ad+bc\le 1+a+b<c+d+a+b<a+b+c+d+ab.
\end{equation*}
This proves that the ratio is always less that 1. This proves the claim and with it theorem \ref{thmpentagon}.
\end{proof}
\end{claim}
\end{section}

\begin{section}{\bf The Hexagon Case}

As in the previous sections we start by fixing a constant $m$ in
$(0,\,1)$ and considering $K=ABCDEF$ an arbitrary convex hexagon. As
before, $K'$ denotes the first $m$-Kasner descendant of $K$. The
main result of this section is given in the following
\begin{thm}\label{thmhexagon}
With the notations above we have that
\begin{equation}
1-2m(1-m)<\frac{\Delta(K')}{\Delta(K)}<1.
\end{equation}
Moreover, both lower and upper bounds are the best possible.
\end{thm}
We need a setup similar to the one used for pentagons. Suppose first
that the long diagonals, $AD$, $BE$, and $CF$ are not concurrent. If
these diagonals do have a common point, then perturb the position of
one of the vertices by an arbitrarily small amount so that the
diagonals are not concurrent anymore. By continuity, any inequality
which is valid in latter case is also valid in the former. Let
$M=AD\cap BE,\,N=AD\cap CF,\,P=CF\cap BE$. Denote
$\mathbf{v}_1=\overrightarrow{MN}, \mathbf{v}_2=\overrightarrow{MP}$
and $\mathbf{v}_3=\overrightarrow{NP}$. It follows that
$\mathbf{v}_3=\mathbf{v}_2-\mathbf{v}_1$ - see figure \ref{tdk5}.

\begin{figurehere}
\begin{center}
\scalebox{.5}{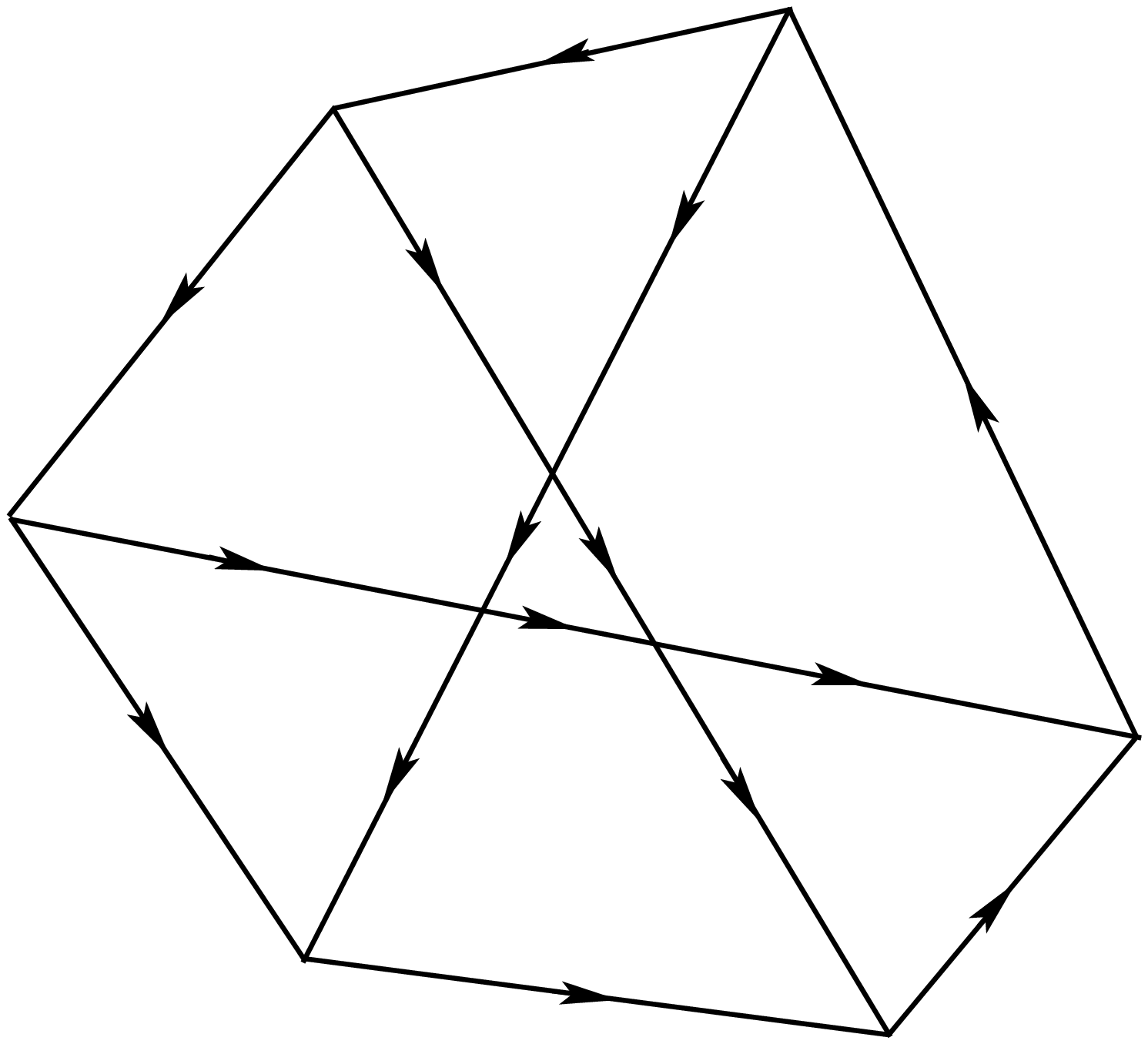} \caption{Setup for the convex
hexagon problem } \label{tdk5}
\end{center}
\end{figurehere}

With appropriate scaling we may assume that
$\Delta(MNP)=\mathbf{v}_1\wedge\mathbf{v}_2=\mathbf{v}_1\wedge\mathbf{v}_3=\mathbf{v}_2\wedge\mathbf{v}_3=
1$. Since $A,\,M,\,N,\,D$ are collinear,
$\overrightarrow{AM}=a\mathbf{v}_1,
\overrightarrow{ND}=d\mathbf{v}_1$ with $a,\,d >0$. Similarly,
$\overrightarrow{BM}=b\mathbf{v}_2,
\overrightarrow{CN}=c\mathbf{v}_3,
\overrightarrow{PE}=e\mathbf{v}_2,
\overrightarrow{PF}=f\mathbf{v}_3$ with $b,\,c,\,e,\,f$ positive
constants.

We try to express $\Delta(K)$, the area of the hexagon in terms of $a, \,b,\,c,\,d,\,e,\,f$.
We begin by computing the areas of the triangles determined by one side and two long diagonals.
\begin{eqnarray*}
\Delta(ABM)&=&\overrightarrow{AB}\wedge\overrightarrow{BM}=(a\mathbf{v}_1-b\mathbf{v}_2)\wedge b\mathbf{v}_2=ab(\mathbf{v}_1\wedge\mathbf{v}_2)=ab.\\
\Delta(CDN)&=&\overrightarrow{CD}\wedge\overrightarrow{CN}=(c\mathbf{v}_3+d\mathbf{v}_1)\wedge c\mathbf{v}_3=cd(\mathbf{v}_1\wedge\mathbf{v}_3)=cd.\\
\Delta(EFP)&=&\overrightarrow{PE}\wedge\overrightarrow{PF}=e\mathbf{v}_2\wedge f\mathbf{v}_3=ef(\mathbf{v}_2\wedge\mathbf{v}_3)=ef.\\
\Delta(BCP)&=&\overrightarrow{BC}\wedge\overrightarrow{BP}=(\mathbf{v}_1+b\mathbf{v}_2-c\mathbf{v}_3)
\wedge(b+1)\mathbf{v}_2=\\
&=&(b+1)(\mathbf{v}_1\wedge\mathbf{v}_2)-c(b+1)(\mathbf{v}_3\wedge\mathbf{v}_2)=1+b+c+bc.\\
\Delta(DEM)&=&1+d+e+de\quad{\mbox{and}}\quad \Delta(FAN)=1+f+a+fa\quad {\mbox {follow similarly}}.
\end{eqnarray*}
Since $\Delta(K)=\Delta(ABM)+\Delta(CDN)+\Delta(EFP)+\Delta(BCP)+\Delta(DEM)+\Delta(FAN)-2\Delta(MNP)$
by using the expressions above and the fact that $\Delta(MNP)=1$ we get that

\begin{equation}\label{ABCDEF}
\Delta(K)=1+(a+b+c+d+e+f)+(ab+bc+cd+de+ef+fa).
\end{equation}
Next let us compute the areas of the ears of the hexagon $ABCDEF$.
Only the computation for the first triangle is shown in detail; the
others can be obtained via circular permutations.
\begin{eqnarray*}
\Delta(ABC)&=&\overrightarrow{AB}\wedge\overrightarrow{BC}
=(a\mathbf{v}_1-b\mathbf{v}_2)\wedge(\mathbf{v}_1+b\mathbf{v}_2-c\mathbf{v}_3)=\\
&=&ab(\mathbf{v}_1\wedge\mathbf{v}_2)-ac(\mathbf{v}_1\wedge\mathbf{v}_3)
-b(\mathbf{v}_2\wedge\mathbf{v}_1)+bc(\mathbf{v}_2\wedge\mathbf{v}_3)=\\
&=&ab-ac+b+bc=b(1+a+c)-ac.\\
\Delta(BCD)&=&c(1+b+d)-bd;\quad \Delta(CDE)=d(1+c+e)-ce.\\
\Delta(DEF)&=&e(1+d+f)-df;\quad \Delta(EFA)=f(1+e+a)-ea.\\
\Delta(FAB)&=&a(1+b+f)-fb.
\end{eqnarray*}
At this point let us introduce a few simplifying notations
\begin{eqnarray}\label{STU}
S:&=&a+b+c+d+e+f,\,\,T:=ab+bc+cd+de+ef+fa\\
U:&=&ac+bd+ce+df+ea+fb.
\end{eqnarray}
It follows that the sum of the areas of all the ears
\begin{equation}\label{sumears}
\Delta(ABC)+\Delta(BCD)+\Delta(CDE)+\Delta(DEF)+\Delta(EFA)+\Delta(FAB)=S+2T-U
\end{equation}
while using (\ref{ABCDEF}) the area of the initial hexagon can be written as
$\Delta(K)=1+S+T.$

If $K'$ denotes the first $m$-Kasner descendant of the hexagon $K$,
then a reasoning identical to the one that lead to equality (\ref
{ear_reasoning}) implies that

$\Delta(K')=\Delta(K)-r\left[\Delta(ABC)+\Delta(BCD)+\Delta(CDE)+\Delta(DEF)+\Delta(EFA)+\Delta(FAB)\right]$.

Using the last five equalities after a few straightforward algebraic manipulations we obtain that
\begin{equation}\label{ratiohex}
\frac{\Delta'}{\Delta}=(1-2r)+r\cdot\frac{2+S+U}{1+S+T}.
\end{equation}
It follows that theorem \ref {thmhexagon} will be proved as soon as we can show that
\begin{claim}
With the notations from the current section we have
\begin{equation}
0<\frac{2+S+U}{1+S+T}<2
\end{equation}and none of the above inequalities can be improved.
\end{claim}
\begin{proof} It is obvious that the ratio is greater than $0$ as $a,\,b,\,c,\,d,\,e,\,f$ are all positive.
To show it can be arbitrarily close to $0$ take $a=b=c=d=1$ and $e=f=n$. It is easy to check that for these
choices the resulting hexagon is convex for all values of $n\ge 1$. Moreover,
\begin{equation*}
\frac{2+S+U}{1+S+T}=\frac{6n+8}{n^2+4n+8}\rightarrow 0 \,\,{\mbox {as}}\,\,n\rightarrow\infty.
\end{equation*}
To prove that the ratio can be arbitrarily close to $2$ take $a=b=c=d=e=f$. Again, it is simple to verify that the resulting hexagon is convex for any value of $a>0$. We have that,
\begin{equation*}
\frac{2+S+U}{1+S+T}=\frac{2+6a+6a^2}{1+6a+6a^2} \rightarrow 2 \,\,{\mbox {as}}\,\,a \rightarrow 0.
\end{equation*}
Finally, since $\Delta(K')<\Delta(K)$, from (\ref{ratiohex}) it follows that the ratio $(2+S+U)/(1+S+T)<2$.

This proves the claim and with it theorem \ref{thmhexagon}.
\end{proof}
The following result is going to be needed later. By the symmetry of
figure \ref {tdk5}, we may assume that $a=\min
\{a,\,b,\,c,\,d,\,e,\,f\}$. Then the following is true
\begin{equation}\label{FABC}
2\cdot\overrightarrow{FA}\wedge\overrightarrow{AB}\le\overrightarrow{EF}\wedge\overrightarrow{AB}+
\overrightarrow{FA}\wedge\overrightarrow{BC}.
\end{equation}
It is easy to check that
$\overrightarrow{FA}\wedge\overrightarrow{AB}=a+ab+af-bf$,
$\overrightarrow{EF}\wedge\overrightarrow{AB}=ae-af+bf$ and
$\overrightarrow{FA}\wedge\overrightarrow{BC}=1+c+f-ab+ac+bf$. Then,
after some algebra, inequality (\ref{FABC}) becomes equivalent to
\begin{equation*}
0\le 1-2a+c+f-3ab+ac+ae-3af+4bf.
\end{equation*}
Since $a=\min \{a,\,b,\,c,\,d,\,e,\,f\}$ we may express $b=a+x_1$,
$c=a+x_2$, $e=a+x_3$ and $f=a+x_4$, where the $x_i$-s are
nonnegative numbers. The last inequality is then equivalent to
\begin{equation*}
0\le 1+x_2+x_4+a(x_1+x_3+x_3+x_4)+4x_1x_4
\end{equation*}
which is obviously true. This proves inequality (\ref{FABC}).

\end{section}

\begin{section}{\bf The Case of the Convex $n$-gon when $n\ge 7$}

Let $m$ be a fixed constant in $(0,\,1)$ and let $K=A_1A_2\ldots
A_n$ be a convex $n$-gon, with $n\ge 7$. Let $K'=B_1B_2\ldots B_n$
be the first $m$-Kasner descendant of $K$, that is, for every
$i=1\ldots n$ point $B_i$ is lies along side $A_iA_{i+1}$ such that
$A_iB_i:B_iA_{i+1}=m:(1-m)$ . The main result of this section is
given by the following
\begin{thm}\label{thmngon}
With the above notations we have that
\begin{equation}
1-2m(1-m)<\frac{\Delta(K')}{\Delta(K)}<1
\end{equation}
and none of the above inequalities can be improved.
\end{thm}
While this result was to be expected (given the statement of theorem \ref {thmhexagon}) , a rigorous proof
still requires some work and inspiration.  We are going to need a couple of intermediate results.

\begin{lemma}\label{lemma1}
Consider a positively oriented convex $n$-gon $K=A_1A_2A_3\ldots
A_n$, $n\geq6$, and denote
$\overrightarrow{A_iA_{i+1}}=\mathbf{v}_i$. Then there exists four
consecutive sides $\mathbf{v}_i$, $\mathbf{v}_{i+1}$,
$\mathbf{v}_{i+2}$ and $\mathbf{v}_{i+3}$ such that
\begin{equation}\label{1234}
\mathbf{v}_{i+1}\wedge \mathbf{v}_{i+2}\leq  \mathbf{v}_{i}\wedge
\mathbf{v}_{i+2}+\mathbf{v}_{i+1}\wedge \mathbf{v}_{i+3}.
\end{equation}
\end{lemma}
\begin{proof}

The case when $K$ is a hexagon has already been proved at the end of
the previous section. In fact, after an appropriate relabeling,
($ABCDEF$ becomes $A_3A_4A_5A_6A_1A_2$) inequality (\ref {FABC})
states that the stronger inequality $2\mathbf{v}_{2}\wedge
\mathbf{v}_{3}\leq  \mathbf{v}_{1}\wedge
\mathbf{v}_{3}+\mathbf{v}_{2}\wedge \mathbf{v}_{4}$ holds true.
Recall that $\mathbf{v}_2\wedge\mathbf{v}_3=\Delta(A_2A_3A_4)>0$ by
convexity.

Suppose now that $n\ge 7$ and denote by $a_{ij}=\mathbf{v}_i\wedge
\mathbf{v}_j$ for all $1\le i,\,j\le n$. We need to show that
\begin{equation}\label{i}
a_{i+1,i+2}\le a_{i,i+2}+a_{i+1,i+3} \quad {\mbox {for some}}\quad
i=1\ldots n.
\end{equation}

{\bf Case 1.}
Suppose that $\mathbf{v}_i\wedge\mathbf{v}_{i+3}\geq0$ for all $i=1\ldots n$\\
Since $a_{i,i+3}\geq 0$, it follows that $a_{i,i+2}>0$ for all
$i=1\ldots n$. Recall that $a_{i,i+1}>0$ by convexity.

In particular
\begin{equation}\label{cond}
a_{13}>0,\,a_{14}\ge 0,\,a_{24}>0,\,a_{25}\ge
0,\,a_{35}>0,\,a_{46}\ge 0.
\end{equation}

Suppose for the sake of contradiction that inequality (\ref {i})
does not hold for $i=1$ or $i=3$. Given (\ref{cond}) this means that
$a_{13}+a_{24}<a_{23}$ from which $0<a_{24}<a_{23}$. Similarly,
$a_{35}+a_{46}<a_{45}$, that is, $ 0<a_{35}<a_{45}.$ Multiplying the
last two inequalities term by term we obtain that
$a_{24}a_{35}<a_{23}a_{45}.$

But Pl\"{u}cker's identity (\ref {plucker}) applied to the indices
$2$, $3$, $4$ and $5$ gives
$a_{23}a_{45}-a_{24}a_{35}+a_{25}a_{34}=0$.

Combining the last two relations it follows that $a_{25}a_{34}<0$
which contradicts (\ref {cond}).

{\bf Case 2.} Suppose that $\mathbf{v}_i\wedge\mathbf{v}_{i+3}<0$
for some $i$ in $\{1, \,2,\,\ldots,\,n\}$. With no loss of
generality say $\mathbf{v}_{n-2}\wedge\mathbf{v}_1< 0$. Then all the
vectors $\mathbf{v}_1$, $\mathbf{v}_2\ldots$, $\mathbf{v}_{n-2}$
belong to the same half-plane - see figure \ref{tdk6}.

If $n\ge 8$, this means that all the vectors $\mathbf{v}_i$, with
$1\le i\le 6$ belong to the same half-plane. This implies that
$\mathbf{v}_i\wedge \mathbf{v}_j >0$ for all $1\le i<j\le 6$ and
therefore all the conditions from (\ref {cond}) are satisfied. Now
we can just repeat the reasoning from case 1 to obtain the desired
conclusion.
\begin{figurehere}
\begin{center}
\scalebox{.51}{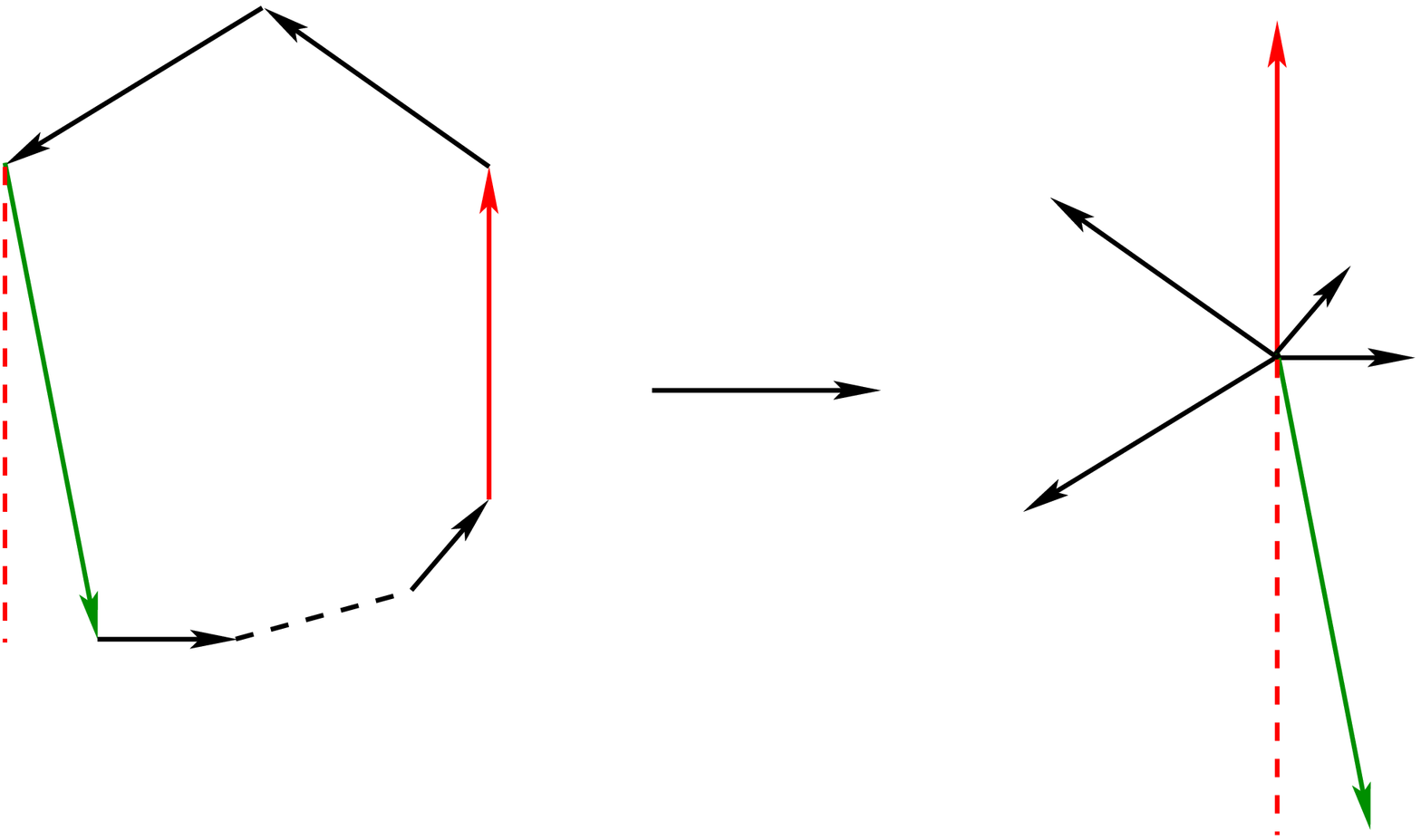} \caption{Lemma 7.2, Case 2,
$n\ge 8$} \label{tdk6}
\end{center}
\end{figurehere}
It remains to see what happens if $n=7$. We still have all the
vectors $\mathbf{v}_i$, with $1\le i\le 5$ lying in the same
half-plane - see figure \ref{tdk7}.

\begin{figurehere}
\begin{center}
\scalebox{.5}{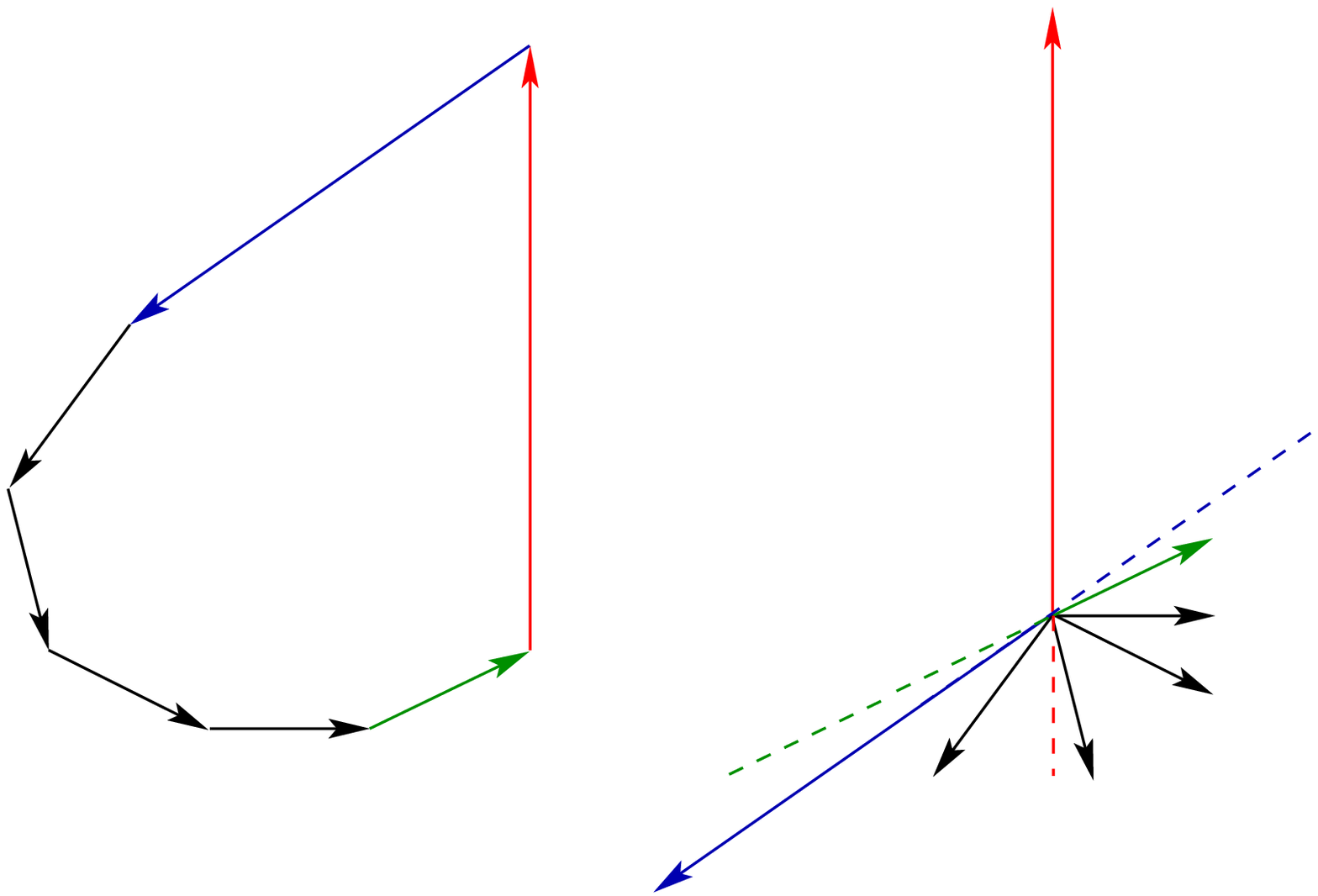} \caption{Lemma 7.2, Case 2,
$n=7$} \label{tdk7}
\end{center}
\end{figurehere}
This means that $\mathbf{v}_i\wedge \mathbf{v}_j >0$ for all $1\le
i<j\le 5$. If $\mathbf{v}_4\wedge \mathbf{v}_6 \ge 0$ then all the
inequalities from (\ref {cond}) are satisfied and we are done. If
$\mathbf{v}_4\wedge \mathbf{v}_6 < 0$ this implies that we have six
consecutive vectors - $\mathbf{v}_6$, $\mathbf{v}_7$,
$\mathbf{v}_1$, $\mathbf{v}_2$, $\mathbf{v}_3$, $\mathbf{v}_4$,
lying in the same half-plane. But this case has been dealt with a
bit earlier.

\end{proof}
We need one more result before we can proceed with the proof of
theorem \ref {thmngon}
\begin{lemma}\label{lemma2}
Let $m$ in $(0,\,1)$ be a fixed constant and let $K=A_1A_2\ldots
A_nA_{n+1}$ be a positively oriented convex $(n+1)$-gon, $n\ge 6$.
Let $K'=B_1B_2\ldots B_nB_{n+1}$ be the first $m$-Kasner descendant
of $K$. Then there exists a convex $n$-gon $L$, obtained by removing
a certain vertex of $K$, such that
\begin{equation}\label{KL}
\frac{\Delta(K')}{\Delta(K)}\ge\frac{\Delta(L')}{\Delta(L)}
\end{equation}
where $L'$ is the first $m$-Kasner descendant of $L$.
\end{lemma}
\begin{proof}
As before, denote $A_iA_{i+1}=\mathbf{v}_i$ and
$\mathbf{v}_i\wedge\mathbf{v}_j=a_{ij}$ for all $1\le i,\,j\le n+1$.
By Lemma \ref{lemma1} we may assume that $a_{23}\le a_{13}+a_{24}$.

Let $L$ be obtained from $K$ after removing vertex $A_3$, that is
$L=A_1A_2A_4\ldots A_nA_{n+1}$. Let point $C$ on $A_2A_4$ such that
$A_2C:CA_4=m:(1-m)$ - see figure \ref{tdk8}. Then, the first
$m$-Kasner descendant of $L$ is $L'=B_1CB_4B_5\ldots B_nB_{n+1}$. It
is easy to see that
$\Delta(L)=\Delta(K)-\Delta(A_2A_3A_4)=\Delta(K)-a_{23}$. On the
other hand, the area of $K'$ exceeds the area of $L'$ by the area of
the non-convex pentagon $P=B_1B_2B_3B_4C$, hence
$\Delta(L')=\Delta(K')-\Delta(P)$.

\medskip

\begin{figurehere}
\begin{center}
\scalebox{.7}{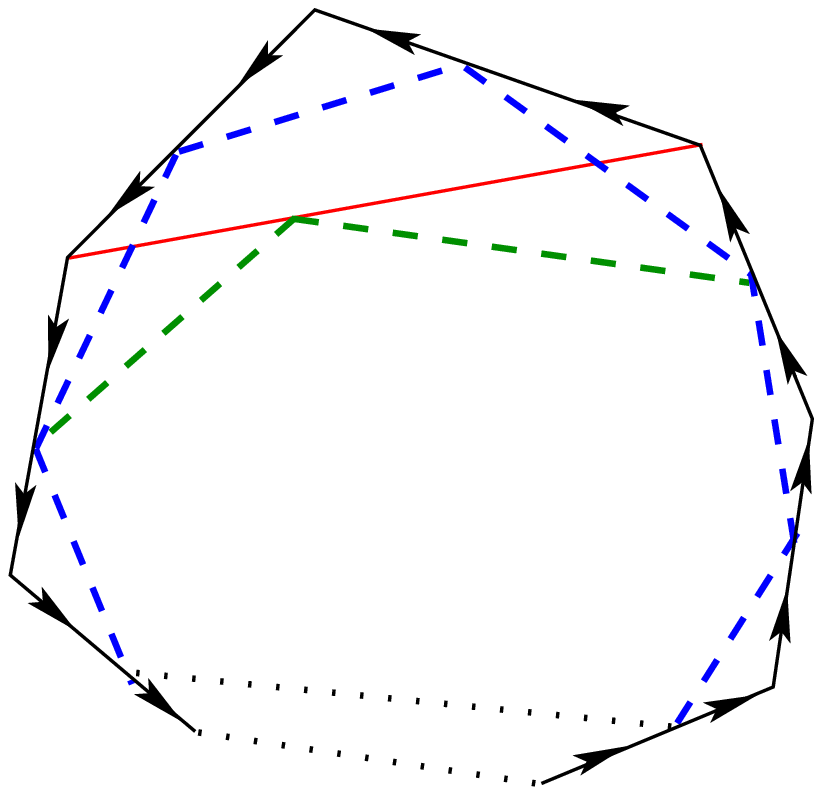} \caption{Figure for lemma 7.3
}\label{tdk8}
\end{center}
\end{figurehere}We first compute $\Delta(P)$. It is easy to see that
\begin{equation}\label{WP}
\Delta(P)=\Delta(B_1B_2C)+\Delta(B_2B_3C)+\Delta(B_3B_4C).
\end{equation}
We have that
\begin{eqnarray*}
\Delta(B_1B_2C)&=&\overrightarrow{B_1B_2}\wedge\overrightarrow{B_2C}=
((1-m)\mathbf{v}_1+m\mathbf{v}_2)\wedge m\mathbf{v}_3=m(1-m)a_{13}+m^2a_{23}.\\
\Delta(B_2B_3C)&=&\overrightarrow{B_3C}\wedge\overrightarrow{B_2C}=
(1-m)\mathbf{v}_2\wedge m\mathbf{v}_3=m(1-m)a_{23}.\\
\Delta(B_3B_4C)&=&\overrightarrow{CB_3}\wedge\overrightarrow{B_3B_4}=
(1-m)\mathbf{v}_2\wedge
((1-m)\mathbf{v}_3+m\mathbf{v}_4)=(1-m)^2a_{23}+m(1-m)a_{24}.
\end{eqnarray*}

Combining the last three equalities into (\ref{WP}) we obtain that
$\Delta(P)=m(1-m)(a_{13}+a_{24}-a_{23})+a_{23}$ and after using our
assumption $a_{23}\le a_{13}+a_{24}$ we have that
\begin{equation}\label{23}
\Delta(P)\ge a_{23}.
\end{equation}
Finally, using (\ref{23}) we obtain that
\begin{equation*}
\frac{\Delta(L')}{\Delta(L)}=\frac{\Delta(K')-\Delta(P)}{\Delta(K)-a_{23}}\le\frac{\Delta(K')}{\Delta(K)}
\end{equation*}
which proves the lemma.
\end{proof}

We are now in position to prove the main result of this section.
Below we give a more precise formulation of theorem \ref{thmngon}.
As above, given $m$ in $(0,1)$ and a convex polygon $K$, $K'$
denotes the first $m$-Kasner descendant of $K$.

\begin{thm}\label{precise}
{\bf i.} For any $m$ in $(0,\,1)$ and for any convex $n$-gon $K$
with $n\ge 6$, we have that
\begin{equation}\label{parti}
1-2m(1-m)<\frac{\Delta(K')}{\Delta(K)}<1.
\end{equation}
{\bf ii.} For any $m$ in $(0,\,1)$, for any positive integer $n\ge
6$ and for any $\epsilon>0$ there exists a convex $n$-gon $K$ such
that
\begin{equation}\label{partii}
\frac{\Delta(K')}{\Delta(K)}<1-2m(1-m)+\epsilon.
\end{equation}
{\bf iii.} For any $m$ in $(0,\,1)$, for any positive integer $n\ge
6$ and for any $\epsilon>0$ there exists a convex $n$-gon $K$ such
that
\begin{equation}\label{partiii}
\frac{\Delta(K')}{\Delta(K)}>1-\epsilon.
\end{equation}
\end{thm}
\begin{proof}{\bf i.}
The second inequality in (\ref{parti}) is trivial since
$int(K')\subset int(K)$. For the first inequality we are going to do
induction on $n$. We have already shown in theorem \ref {thmhexagon}
that the first inequality is true if $n=6$. Let $K$ be a convex
$(n+1)$-gon, $n\ge 6$. Then according to lemma \ref{lemma2} one can
remove a vertex of $K$ such that the resulting $n$-gon $L$ has the
property stated in (\ref{KL}). Coupling this with the induction
hypothesis we obtain that
\begin{equation*}
\frac{\Delta(K')}{\Delta(K)}\ge\frac{\Delta(L')}{\Delta(L)}>1-2m(1-m).
\end{equation*}
{\bf ii.} Start with a triangle of unit area, $MA_1A_2$. Cut of a
small triangle $MA_3A_n$ of area $\epsilon ^2$  as shown in figure
9. Then replace the segment $A_nA_3$ by a small circular arc along
which place the remaining vertices $A_4$, $A_5$, $\ldots$, $A_{n-1}$
- as in the figure \ref{tdk9} below.
\begin{figurehere}
\begin{center}
\scalebox{.5}{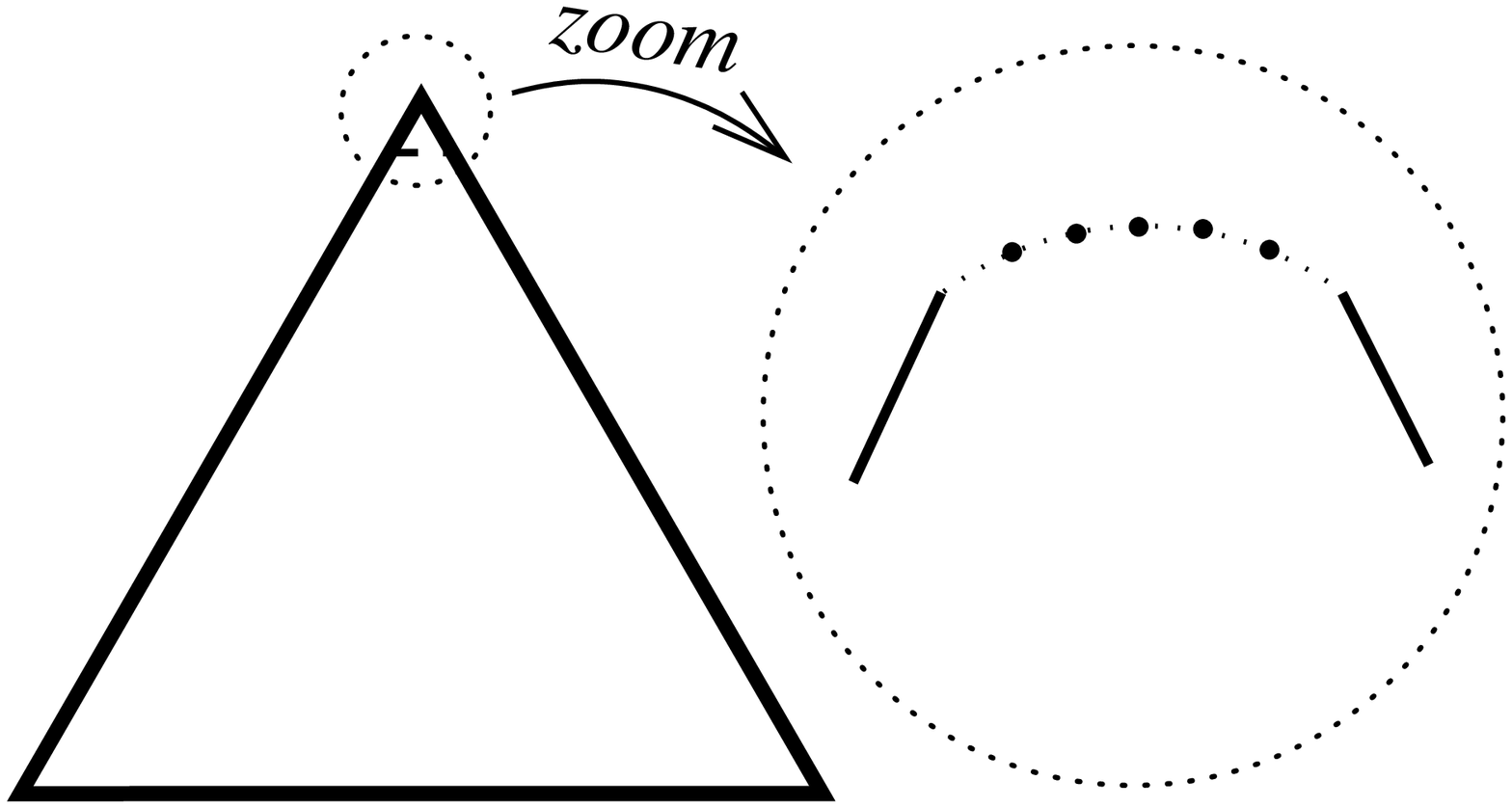} \caption{Main theorem,
part {\bf ii}}\label{tdk9}
\end{center}
\end{figurehere}
We claim that the polygon $K=A_1A_2\ldots A_n$ defined above has the
property (\ref {partii}). Denote $m(1-m)=r$ and let $K'$ be the
first $m$-Kasner descendant of $K$. We have
\begin{equation*}
2(1-\epsilon)=\Delta(A_nA_1A_2)+\Delta(A_1A_2A_3)<
\sum_{i=1}^n\Delta(A_{i-1}A_iA_{i+1})=\frac{\Delta(K)-\Delta(K')}{r}
\end{equation*}
which after we divide by $\Delta(K)<1$ and rearrange the terms
becomes
\begin{equation*}
\frac{\Delta(K')}{\Delta(K)}<1-2r+2r\epsilon\le1-2r+\frac{\epsilon}{2}=1-2m(1-m)+\frac{\epsilon}{2}
\end{equation*}
since $r=m(1-m)\le 1/4$. This proves part {\bf ii.}

{\bf iii.} We will use induction. We already proved that there are
hexagons which satisfy (\ref {partiii}). Let $Q=A_1A_2\ldots A_n$ be
a positively oriented convex $n$-gon, $n\geq6$, for which
$\Delta(Q')/\Delta(Q)>1-\epsilon$. Without loss of generality assume
that $a_{n-1,1}>0$. Construct a point $A_{n+1}$ such that
$\overrightarrow{A_nA_{n+1}}=\lambda(\mathbf{v}_{n-1}+\mathbf{v}_n)$,
where $\lambda<\min\{1/2,a_{n,1}/(a_{n,1}+a_{n-1,1})\}$. Then,
$P=A_1A_2A_3\ldots A_nA_{n+1}$ is a positively oriented convex
$(n+1)$-gon as shown in figure \ref{tdk10} below.
\begin{figurehere}
\begin{center}
\scalebox{.40}{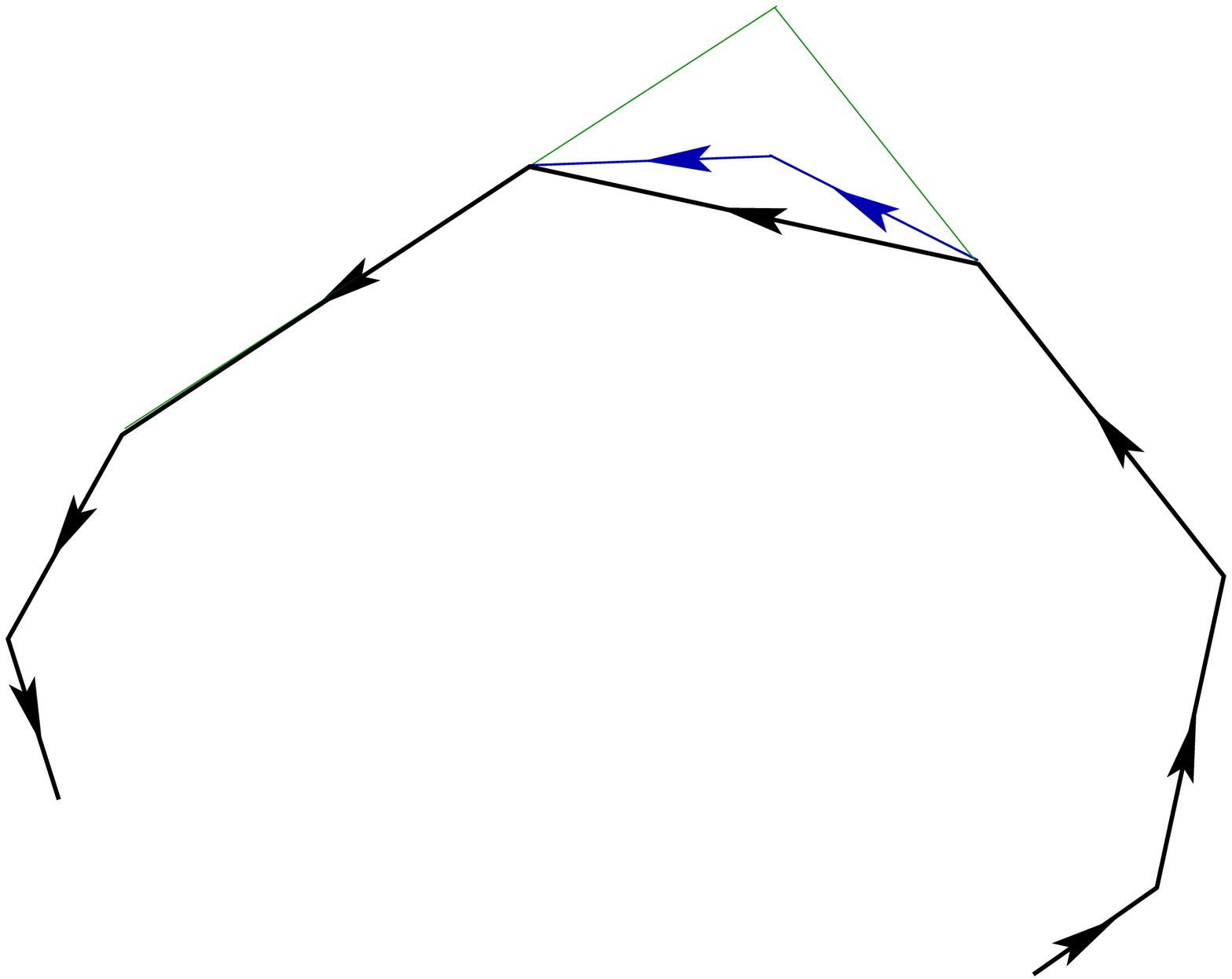} \caption{Main theorem, part
{\bf iii}}\label{tdk10}
\end{center}
\end{figurehere}

We claim that:
$\quad\overrightarrow{A_{n-1}A_n}\wedge\overrightarrow{A_{n+1}A_1}+
\overrightarrow{A_nA_{n+1}}\wedge\overrightarrow{A_1A_2}\geq\overrightarrow{A_nA_{n+1}}\wedge\overrightarrow{A_{n+1}A_1}.$

Indeed, this is equivalent to
\begin{eqnarray*}
&&\mathbf{v}_{n+1}\wedge(\mathbf{v}_n-\lambda(\mathbf{v}_{n-1}+\mathbf{v}_n))+\lambda(\mathbf{v}_{n-1}+\mathbf{v}_n)\wedge\mathbf{v}_1\geq\lambda(\mathbf{v}_{n-1}+
\mathbf{v}_n)\wedge(\mathbf{v}_n-\lambda(\mathbf{v}_{n-1}+\mathbf{v}_n))\Leftrightarrow\\
&&\Leftrightarrow a_{n-1,n}-\lambda a_{n-1,n}+\lambda
a_{n-1,1}+\lambda a_{n,1}\geq\lambda a_{n-1,n}\Leftrightarrow
a_{n-1,n}\geq\lambda(2a_{n-1,n}-a_{n-1,1}-a_{n,1})
\end{eqnarray*}
which is true since we have $a_{n-1,1}>0$ by assumption, $a_{n,1}>0$
by convexity and $\lambda\le 1/2$. It follows that the hypotheses
from lemma \ref {lemma2} are valid for polygon $P$ and vertex
$A_{n+1}$, that is, we have constructed a convex $n+1$-gon $P$ for
which
\begin{equation*}
\frac{\Delta(P')}{\Delta(P)}\ge\frac{\Delta(Q')}{\Delta(Q)}>1-\epsilon.
\end{equation*}
This completes the proof of theorem \ref {precise}.

\end{proof}
\end{section}

{\bf Conclusions and Further Research.}

In the present paper we provide a complete answer regarding the
ratio between the area of a convex polygon and the area of its first
$m$-Kasner descendent. It would be interesting to extend these
results to the ratio between $\Delta(K)$, the area of the original
polygon, and $\Delta(K^{t})$, the area of its $t$-th $m$-Kasner
descendant. Same questions can be asked if instead of areas one
considers perimeters.


\begin{thebibliography}{99}


\bibitem{BS} F. Bachmann and E. Schmidt, {\it $N$-gons}, Univ. of
Toronto Press, 1975. translated by C. W. L. Garner.

\bibitem{C} R. J. Clarke, Sequences of Polygons, {\it Mathematics
Magazine}, {\bf {29}}(1979), pp. 102--105.

\bibitem{D} P. J. Davis, {\it Circulant Matrices}. Chelsea Pub Co,
1994, second edition.

\bibitem{HZ} R. Hitt and X. M. Zhang. Dynamic Geometry of Polygons.
{\it Elemente der mathematik}, {\bf {56}}(2001), pp. 21-37.

\bibitem{K} E. Kasner. The group generated by central symmetries,
with applications to polygons. {\it American Mathematical Monthly},
{\bf{10}}(1903), pp. 57--63. reprinted in Selected Papers on
Algebra, Mathematical Association of America, 1977, pp. 66--71.

\bibitem{KN} E. Kasner and J. Newman, {\it Mathematics and the
Imagination}, Simon and Schuster, New York, 1940.

\bibitem{KM} P. K. Kelly and D. Merriell, Concentric polygons, {\it
American Mathematical Monthly}, {\bf{71}}(1964), pp. 37--41.

\bibitem{L} G. L\"{u}k\~{o}. Certain sequences of inscribed
polygons. {\it Periodica Mathematica Hungarica}, {\bf{3}}(1973), no.
3-4, pp. 255--260.

\bibitem{LFT} L. F. T\'{o}th. Iteration methods for convex polygons.
{\it Mathematikai Lapok}, {\bf {20}}(1969), pp. 15--23.

\bibitem{Sch} I. J. Schoenberg, {\it Mathematical Time Exposures}.
Ma thematical Association of America, Washington D.C., 1982.

\end{thebibliography}
\end{document}